\documentclass[11pt, twoside]{article}
\usepackage{amsfonts,amssymb,amsmath,amsthm,dsfont}
\usepackage{verbatim}
\usepackage{graphicx}
\usepackage[all]{xy}
\usepackage{xmpmulti,amscd,color,pstricks}

 \divide\oddsidemargin by 2
\newtheorem{thm}{Theorem}[section]

\newtheorem{cor}[thm]{Corollary}
\newtheorem{lem}[thm]{Lemma}
\newtheorem{prop}[thm]{Proposition}

\theoremstyle{definition}

\newtheorem{rem}[thm]{Remark}
\newtheorem*{rem*}{Remark}

\numberwithin{equation}{section}

\definecolor{OrangeRed}{cmyk}{0,0.6,1,0}            
\definecolor{DarkBlue}{cmyk}{1,1,0,0.20}
\definecolor{DarkGreen}{cmyk}{1,0,0.6,0.2}
\definecolor{myblue}{rgb}{0.66,0.78,1.00}
\definecolor{Violet}{cmyk}{0.79,0.88,0,0}
\definecolor{Lavender}{cmyk}{0,0.48,0,0}

\renewcommand{\Im}{\operatorname{Im}}
\renewcommand{\Re}{\operatorname{Re}}

\newcommand{\dist}{\operatorname{dist}}

\newcommand{\MM}{{\cal M}}
\newcommand{\NN}{{\cal N}}

\newcommand{\PP}{{\cal P}}

\newcommand{\XX}{{\cal X}}
\newcommand{\ZZ}{{\cal Z}}

\newcommand{\C}{{\mathbb C}}

\newcommand{\D}{{\mathbb D}}
\renewcommand{\H}{{\mathbb H}}

\newcommand{\N}{{\mathbb N}}

\renewcommand{\P}{{\mathbb P}}
\newcommand{\R}{{\mathbb R}}

\newcommand{\ra}{\rightarrow}
\newcommand{\ov}{\overline}

\renewcommand{\epsilon}{\varepsilon}
\renewcommand{\phi}{\varphi}

\begin{document}

\title{Invariant escaping Fatou components with two rank 1 limit functions for automorphisms of $\C^2$}

\author{Anna Miriam Benini, Alberto Saracco, Michela Zedda\thanks{This  project has been partially supported by:  The project 'Transcendental Dynamics 1.5' inside the program FIL-Quota Incentivante of the  University of Parma and co-sponsored by Fondazione Cariparma; Indam through the research groups GNAMPA and GNSAGA; PRIN 2017 'Real and Complex Manifolds: Topology, Geometry and holomorphic dynamics'.}}
\date{\today}
\maketitle
\begin{abstract}We construct automorphisms of $\C^2$, and more precisely  transcendental H\'enon maps, with an invariant escaping  Fatou component which has exactly two distinct limit functions, both of (generic) rank 1.   We also prove a general growth lemma for the norm of points  in orbits belonging to invariant escaping Fatou components for   automorphisms of the form $F(z,w)=(g(z,w),z)$ with $g(z,w):\C^2\ra\C$ holomorphic.
\end{abstract}
\section{Introduction} 

We consider the dynamical system generated by the iteration of a holomorphic automorphism $F:\C^2\ra\C^2$. A \emph{Fatou component} is a maximal connected open set $U$ on which the family of iterates $\{F^n\}$ is normal, that is, every sequence has a  subsequence which converges uniformly on compact sets to a holomorphic function $g:U\ra\P^2$, where $g$ may depend on the subsequence itself (see \cite{henon1} for a discussion about the definition of normality). Such a function $g$ is called a \emph{limit function}, and  its image $g(U)$ is called  a \emph{limit set}. 
If a limit set intersects the line at infinity, then it is in fact contained in it (see Lemma 2.4 and 4.3 in \cite{henon1}).

 It is natural to  classify invariant Fatou components both from the point of view of a dynamical characterization (that is, to which model map the iterates  are conjugate to) and from the point of view of a  geometric characterization (that is, to which model manifold the Fatou component is biholomorphic). The first characterization strongly influences the latter, for example, for polynomial automorphsims of $\C^2$,  any  invariant Fatou component on which the iterates converge to a fixed point is biholomorphic to $\C^2$ \cite{UedaLocalStructure, PVW,RosayRudin}. The dynamical characterization is also very related to which types of limit functions there can be in the Fatou component,  for example, their rank, and whether the limit sets  are in the boundary of the Fatou component or in its interior.

In this paper we consider invariant escaping Fatou components. A Fatou component $U$ is called \emph{escaping} if for any 
 of its limit functions  $g$ we have $g(U)\subset \ell^\infty$, where $\ell^\infty$ is the line at infinity in the projective space $\P^2$ used to compactify $\C^2$. 

In the past three decades, the  investigation of the dynamics of holomorphic maps from $\C^2$ to $\C^2$ has concentrated on studying polynomial automorphisms, and in particular (polynomial) H\'enon maps, that is   automorphisms with constant Jacobian of the form $$F(z,w)=(P(z)+\delta w, z)$$
with $P:\C\ra\C$ polynomial of degree $d\geq2$. Indeed, by results of Friedland and Milnor \cite{FM89}, any polynomial automorphism is conjugate to either an elementary map or a finite composition of polynomial H\'enon maps, so studying the latter, gives a relatively complete picture of the dynamics of polynomial automorphisms of $\C^2$.  
For polynomial H\'enon maps it is not difficult to see \cite{BS91} that unbounded forward orbits belong to the Fatou set and   converge to the point $[1:0:0]\in \ell^\infty$. So in this case, there is always exactly one escaping Fatou component, which  can be seen  as the attracting basin of  $[1:0:0]$, and whose structure has been studied for example in \cite{HOV1,BedfordSmillieRays,Mummert}. 
So for polynomial automorphism, the matter of existence and properties of escaping Fatou components is essentially  settled.

On the other hand, one dimensional transcendental dynamics shows that periodic Fatou components on which the iterates tend to infinity (called \emph{Baker domains} in this setting)  are as of today an active research topic  (see for example the most recent papers \cite{BFJK1,BFJK2,MartiPeteEscaping,BZ12,RempeSingularBaker}). One may be tempted to  think of Baker domains as parabolic basins whose  parabolic fixed  point has been moved to infinity, but in fact,  there can be  different dynamical behaviours (\cite{Cowen,FagHenDeformationBaker}), only some of which   relate to parabolic dynamics.
   On the other hand, from the geometric point of view,  all Baker domains for entire functions are simply connected, and hence, because of the Riemann Uniformization theorem, biholomorphic to the unit disk $\D$.  Inspired by the one-dimensional examples, a transcendental H\'enon map featuring an  escaping Fatou component with a constant limit function and  which is not an attracting basin has been constructed in \cite{henon1}, Section 5. 

Our first preliminary result is that orbits in escaping Fatou components cannot grow too fast  under appropriate conditions.   This is in analogy with results obtained by  Baker 
\cite[Theorem 1]{BakerInfiniteLimits} for  Baker domain in one variable, and in contrast to the escaping points constructed in \cite{henon2}, whose orbits converge to infinity faster than any polynomial.  The proof uses methods similar to  \cite{henon1}, Lemma 5.9. 

\begin{prop}[Slow growth in escaping components]\label{prop:slow growth intro}
Let $F$ be an automorphism of the form $F(z,w)=(g(z,w),z)$ with an  escaping Fatou component   $U$ on which the iterates converge  to a function $h: U\ra\ell^\infty$ uniformly on compact subsets. 

Let   $K$ be  a compact subset of $U$, such that $h$  does not take the values $[0:1:0], [1:0:0]$ on $K$, and fix  $0<\epsilon<\min_K |h|$.   Then there exists $C=C(K)$  such that  for $n$ large enough  and for  any $P\in K$ we have
\begin{equation}
\frac{(\min_K|h|-\epsilon)^n}{C}\leq \|F^n(P)\|\leq C(\max_{K}|h|+\epsilon)^n.
\end{equation}
\end{prop}

 The  main result of this  paper  is the  construction of  examples of transcendental H\'enon maps with an escaping Fatou component which has exactly 2 limit functions, both  of (generic) rank 1.
Transcendental H\'enon maps are  automorphisms with constant Jacobian of the form $$F(z,w)=(f(z)+\delta w, z)$$
with $f\!:\C\ra\C$ entire transcendental. They have been introduced in \cite{Dujardin04} to construct automorphisms with infinite entropy, and have beeen studied in    \cite{henon1, henon2,henon3}. Transcendental H\'enon maps   always have both escaping and periodic points (hence non-empty Julia set), infinite entropy, a pseudoconvex Fatou set, and can exhibit a variety of dynamical behaviour ranging from having various types of wandering domains to the possibility that the Julia set is all of $\C^2$.
 
\begin{thm}[Escaping components with distinct rank 1 limit functions]\label{thm:main theorem intro}
Let $f\!:\mathbb C\rightarrow \mathbb C$ be a transcendental entire function which  is bounded in a right half plane, and $a>1$. 
Let $F\!:\mathbb C^2\rightarrow \mathbb C^2$ be the transcendental H\'enon map  defined by 
$$
F(z,w)=(f(z)+aw,z).
$$ 
Then \begin{enumerate}
\item $F$ has an invariant escaping Fatou component $U$  with exactly two distinct  limit functions $h_1, h_2: U\ra \ell^\infty$, both of which have  (generic)  rank 1.
\item $h_1(U), h_2(U)\supset \ell^\infty\setminus \{[1:0:0],[0:1:0]\}$.
\item $F$ is conjugate to the linear map $L(z,w)=(aw,z)$ on an appropriate subset of $U$.
\item If $f(z)=e^{-z}$, then $F$ is conjugate to $L$ on all of $U$, and $U$ is biholomorphic to $\H\times\H$. 
\end{enumerate} 
\end{thm}
  Several ideas in the proof are taken from \cite{henon1}, Section 5, modified to apply to this  different setting. The escaping component constructed in \cite{henon1} differs from ours both from the dynamical and from the geometric point of view: indeed, it is  biholomorphic to $\H\times\C$ and  the map is conjugate to the linear map $G(z,w)=(2z-w,w)$.
 
 In general, it is very unclear under which conditions and for which types of automorphisms it is  possible to have   invariant Fatou components  with    limit sets of dimension 1   in the boundary.  While   \cite{LP14} gives  conditions under which this cannot happen for polynomial H\'enon maps,  there are a few examples of automorphisms sporting a Fatou component with a rank one   limit manifold in the boundary: see   \cite{JL04,  BTBP, ReppCylinders}. All examples are of non-escaping Fatou components, for   automorphisms with non-constant Jacobian, and wherever this has been computed the map in question is conjugate to the linear map $G(z,w)=(z+1, w)$, so their dynamics can be considered parabolic. On the other hand, their  complex structures are different: the Fatou components in \cite{BTBP} are biholomorphic to $\C^2$ while the ones in \cite{ReppCylinders}  are biholomorphic to $\C^*\times\C$ (compare with the construction in \cite{BRS17}).  
 
 Let us conclude by  remarking that from both the dynamical and the geometric point of view, the richness of possibilities in $2D$ as compared to $1D$ is striking. In one variable, all periodic and preperiodic Fatou components for entire and meromorphic functions are fully classified: on each such component, including Baker domains,  the dynamics is semi-conjugate to an appropriate linear map, and in the entire case, all periodic components are simply connected hence biholomorphic to the unit disk.  
 
In several variables,   recurrent Fatou components for polynomial automorphisms have been classified in \cite{BS91} (see also \cite{henon1}), but it is currently unknown whether such components can be biholomorphic to an   annulus  times $\C$ (for convincing evidence that this may indeed happen see \cite{Bed18}). Several additional geometric   possibilities are open in the transcendental H\'enon case: a priori, the rotation surface may be bioholomorphic also to the punctured disk, the punctured plane, or even the plane itself.  Non-recurrent Fatou components have been classified in \cite{LP14} for polynomial automorphisms under the  assumption that the Jacobian is small, however, removing this assumption it is not known what other dynamical behaviours may appear and what would be the geometry of the limit sets and of the Fatou component. 

\subsection*{Acknowledgements} The first author would like to thank Eric Bedford, Filippo Bracci, and Josias Reppekus for interesting discussions about nonrecurrent and escaping Fatou components for automorphisms of $\C^2$ at the Universit\`a di Roma Tor Vergata.
\subsection*{Notation} We denote by $\C$ the complex plane, by $\hat{\C}=\C\cup \{\infty\}$ its one-point compactification (the Riemann sphere), by $\H$ the right half plane $\{\Re z>0\}$, and by $\D$ the Euclidean unit disk.  The complex projective space is denoted by $\P^2$ and the line at infinity by $\ell^\infty$.
 
The complex line $\ell^\infty$ is biholomorphic to the Riemann sphere $\hat{\C}$, via the biholomorphism $\phi$ which sends $[p:q:0]$ to $\frac{p}{q}$.  Given a holomorphic map $h$ from a domain of $\C^2$ to $\ell^\infty$ we identify it with a holomorphic map $h$ to the Riemann sphere, or  equivalently,  with  a meromorphic map to $\C$. 


\section{Slow growth in escaping components}\label{sect:slow growth} 
Using hyperbolic geometry, Baker 
\cite[Theorem 1]{BakerInfiniteLimits} 
 proved that,  if $z$ is in a Baker domain for an entire transcendental function $f:\C\ra\C$, then
$\log |f^n(z)|= O(n)$ as $n\ra\infty$.
 We show an analogous result for periodic escaping components for transcendental H\'enon maps, establishing Proposition~\ref{prop:slow growth intro}.  This is in contrast to the escaping points constructed in \cite{henon2}, whose orbits converge to infinity  faster than any polynomial.  This result, which we believe to be of independent interest, is used  in Section~\ref{sect:absorbing}.
 
  We restate Proposition~\ref{prop:slow growth intro} for convenience.
   The proof uses   methods similar to  \cite{henon1}, Lemma 5.9.

\begin{prop}[Slow growth in  escaping components]\label{prop:slow growth basic}
Let $F$ be an automorphism of the form $F(z,w)=(g(z,w),z)$ with an  escaping Fatou component   $U$ on which the iterates converge  to a function $h: U\ra\ell^\infty$ uniformly on compact sets. 

Let   $K$ be  a compact subset of $U$, such that $h$  does not take the values $0, \infty$ on $K$, and fix  $0<\epsilon<\min_K |h|$.   Then there exists $C=C(K)$  such that  for $n$ large enough  and for  any $P\in K$ we have
\begin{equation}\label{eqtn:slow growth in Baker}
\frac{\left(\min_K|h|-\epsilon\right)^n}{C}\leq \|F^n(P)\|\leq C\left(\max_{K}|h|+\epsilon\right)^n.
\end{equation}
\end{prop}
\begin{proof}
Let $P_n=(z_n,w_n)=(z_n,z_{n-1})$ by the special form of $F$. Since $F^n(P)\ra\ell^\infty$ and $h$ does not take the values $0, \infty$ we can assume that $|z_n|, |w_n|\neq 0$ for all $n$ large enough. 
   Since $F^n\ra h$ uniformly on $K$,  there exists $n_\epsilon$ such   that for all $n\geq n_\epsilon$   and all $P\in K$ we have
$$
\left|\frac{z_n}{w_n}-h(P)\right|=\left|\frac{z_n}{z_{n-1}}-h(P)\right|<\epsilon,
$$ 
hence, using the triangular inequality,
$$
\min_K|h|-\epsilon \leq \left|\frac{z_n}{z_{n-1}}\right|\leq \max_K |h|+\epsilon,
$$
from which it follows (adding the multiplicative factor $c^{\pm1}$ to account for $n\leq n_\epsilon$)
  $$
  \frac{\left(\min_K|h|-\epsilon\right)^n}{c}\leq |z_n|\leq c\left(\max_{K}|h|+\epsilon\right)^n.
  $$
  Since $w_n=z_{n-1}$, the analogous inequality holds for $|w_n|=|z_{n-1}|$ and the claim for $\|P_n\|$ follows for some constant $C$.
\end{proof}
 
The proof is easily generalized to obtain the following: 

\begin{prop}[Slow growth general version]\label{prop:slow growth general}
Let $F$ be as in Proposition~\ref{prop:slow growth basic} with an escaping Fatou component $U$ with finitely many limit functions $h_i:U\ra\ell^\infty.$ Suppose that there exists a partition of $\N$ into finitely many subsequences $\NN_i$ such that for each $i$ the iterates of $F$ converge to the limit function $h_i$ along the subsequence $\NN_i$. 
Let $K$ be a compact subset of $U$ such that none of the $h_i$ attains the value $0, \infty$ on $K$, and let $\epsilon<\min_{K,i}|h_i|$. Then  for any $P\in K$  and $n$ large enough we have 
\begin{equation}\label{eqtn:slow growth in Baker}
\frac{\left(\min_{K,i}|h_i|-\epsilon\right)^n}{C}\leq \|F^n(P)\|\leq C\left(\max_{K,i}|h_i|+\epsilon\right)^n.
\end{equation}
 \end{prop}
  
We note the following corollary that we will use in Section~\ref{sect:absorbing}, and follows from the proof of Proposition~\ref{prop:slow growth basic}  and Proposition~\ref{prop:slow growth general}.
\begin{cor}\label{cor:slow growth 2}Let $U$ be an escaping Fatou component for $F$ as in Proposition \ref{prop:slow growth basic}, such that $F^{2n}\ra h_1$, $F^{2n+1}\ra h_2$.  Then for any $K$ compact subset of $U$ with $h_i\neq0,  \infty$ on $K$ there exists $C$ such that for $n$ large enough and  every    $P\in K$  we have 

 \begin{equation}\label{eqtn:sandwich growth}
  \frac{\left(\min_K|h_i|-\epsilon\right)^n}{C}\leq |z_n|\leq C\left(\max_K|h_i|+\epsilon\right)^n,
 \end{equation}
 where $(z_n,w_n):=F^{n}(P)$.
\end{cor} 

\section{A transcendental H\'enon map with an invariant escaping Fatou component with two distinct limit functions of rank 1}\label{sect:Distinct limit functions rank 1}
 In this section we construct a family of transcendental H\'enon maps, each of which has an invariant escaping Fatou component with exactly two limit functions, both of which have (generic) rank 1.  Recall that  the \emph{rank} of a holomorphic functions $h$ at a point $P$ is the rank of its differential at $P$.  


\begin{prop}\label{prop:distinct limit functions}
Let  $a>1$ and let $f\!:\mathbb C\rightarrow \mathbb C$ be a nonlinear entire function which  is bounded in a right half plane. 
Let $F\!:\mathbb C^2\rightarrow \mathbb C^2$ be the transcendental H\'enon map  defined by 
$$
F(z,w)=\left(aw+f(z),z\right).
$$ 
Then $F$ has an invariant escaping Fatou component $U$  with exactly two distinct  limit functions $h_1, h_2: U\ra \ell^\infty$, both of which have  generic rank 1 and such that 
$F^{2n}(z,w)\ra h_1(z,w)$, $F^{2n+1}(z,w)\ra h_2(z,w)$ as $n\ra\infty$, uniformly on compact subsets of $U$.
\end{prop}
Observe that the condition that $f$ is nonlinear, entire and bounded in a right half plane  implies that $f$ is transcendental, and that  the function $f(z)=A e^{-kz}$ satisfies the hypothesis of the proposition for every $A,k>0$, as well as any finite linear combination of such function.   
Many more examples, even with the stronger assumption that $|f|\ra0$ as $\Re z\ra\infty$,  can be constructed using   tangential approximation, for example, the following relatively elementary result     [Gai87, Theorem 2'  page 153; see also page 142].  

\begin{thm}[Approximation]
Let $S\subset \C$ be a closed set such that  $\hat{\C}\setminus S$ is connected and locally connected at infinity. Let $h$ be holomorphic in the interior of ${S}$ and   continuous on $\ov{S}$ (the closure of $S$ in $\C$). Let $\epsilon>0$. 
Then there exists $g$ entire such that 
$$|g-h|<\epsilon\ \text{ on $S$ and  }\ |g(z)-h(z)|<\frac{1}{|z|}\ \text{as $|z|\ra\infty$ on $S$}.$$
\end{thm}
Indeed, if we set  $h=0$ on the right half plane and anything you like in, say, a finite collection of topological disks with pairwise disjoint closure (which do not intersect the right half plane ), then the approximating $g$ will satisfy the assumptions of Proposition~\ref{prop:distinct limit functions}.

From the perspective of  the identification of  $\ell^\infty$ with $\hat{\C}$ that associates to the point $[p:q:0]\in\ell^\infty$ the point $\frac{p}{q}\in\hat{\C}$, with $p,q\in\C$, the limit functions $h_i$ are meromorphic functions from $U$ to $\hat{\C}$.

\begin{proof}[Proof of Proposition~\ref{prop:distinct limit functions}]
Given $(z_0,w_0)$ in $\C$ we define  $(z_n,w_n):=F^n(z_0,w_0)$. For $R>0$ we define the set
 \begin{equation}\label{eq:WR}
  W_R:=\left\{(z,w)\in\C^2: \Re z,\Re w>R\right\}.
\end{equation}
Fix $\epsilon>0$. Since  $f$ is bounded in a right half plane, for any  $R$ sufficiently large  we have  that $|f(z)|<(a-1)R-\epsilon$ for all $z$ with $\Re z>R$. Let $W=W_R$ for any $R$  which satisfies this condition. 
Then  for $(z_0,w_0)\in W$ we have 
\begin{align*}
\Re z_1&= a\Re w_0+\Re f(z_0)>aR-R+R-|f(z_0)|>R+\epsilon \\
\Re w_1&=  \Re z_0>R
\end{align*}
and hence $W$  is forward invariant and  $\Re z_n, \Re w_n\ra\infty$ if $z_0,w_0\in W$. It follows that  $F^n(z_0, w_0)\ra\ell^\infty$. 

We now show convergence of the subsequences $F^{2n}$ and $F^{2n+1}$ on $W$, implying that $W$ is contained in an escaping Fatou component. 

A recursive  computation gives that 
\begin{align}
\label{eqtn:recursive even}F^{2n}(z_0,w_0)&=\left(a^nz_0+a^n\sum_{j=1}^na^{-j}f(z_{2j-1}),a^nw_0+a^n\sum_{j=1}^na^{-j}f(z_{2j-2})\right),\\
\label{eqtn:recursive odd}F^{2n+1}(z_0,w_0)&=\left(a^{n+1}w_0+a^{n+1}\sum_{j=1}^{n+1}a^{-j}f(z_{2j-2}),a^nz_0+a^n\sum_{j=1}^na^{-j}f(z_{2j-1})\right).
\end{align}
Consider the ratio 
\begin{equation}\label{eqtn:Miky pari}
\frac{z_{2n}}{w_{2n}}=\frac{a^nz_0+\sum_{j=1}^na^{n-j}f(z_{2j-1})}{a^nw_0+\sum_{j=1}^na^{n-j}f(z_{2j-2})}=\frac{z_0+\sum_{j=1}^na^{-j}f(z_{2j-1})}{w_0+\sum_{j=1}^na^{-j}f(z_{2j-2})},
\end{equation}
 and the ratio
\begin{equation}\label{eqtn:Miky dispari}
\frac{z_{2n+1}}{w_{2n+1}}=\frac{a^{n+1}w_0+a^{n+1}\sum_{j=1}^{n+1} a^{-j}f(z_{2j-2})}{a^{n}z_0+a^{n}\sum_{j=1}^na^{-j}f(z_{2j-1})}= \frac{a(w_0+\sum_{j=1}^{n+1} a^{-j}f(z_{2j-2}))}{z_0+\sum_{j=1}^na^{-j}f(z_{2j-1})}.
\end{equation}
Set 
 \begin{equation}\label{eqtn:Delta}
\Delta:=\max\left(\left|\sum_{j=1}^\infty a^{-j}f\left(z_{2j-1}\right)\right|, \left|\sum_{j=1}^\infty a^{-j}f\left(z_{2j-2}\right)\right|\right ). 
  \end{equation}
 Using the assumption  that $|f(z)|$ is bounded   for $\Re z>R$ we get that 
  \begin{equation}\label{eqtn:DeltaBound}
\Delta \leq \sum_{j=1}^\infty\left|a^{-j/2}f(z_{j})\right|<\sup_{\Re z>R}\left|f(z)\right| \sum_{j=1}^\infty \left|a^{-j/2}\right|<\infty. 
  \end{equation}
 Hence we can take the limit as $n\ra\infty$ in (\ref{eqtn:Miky pari}) and (\ref{eqtn:Miky dispari}) to obtain
\begin{align}\label{eqtn:definition of h1}
h_1(z_0,w_0):=&\lim_{n\rightarrow \infty}\frac{z_{2n}}{w_{2n}}=\frac{z_0+\sum_{j=1}^\infty a^{-j}f(z_{2j-1})}{w_0+\sum_{j=1}^\infty a^{-j}f(z_{2j-2})},\\
h_2(z_0,w_0):=&\lim_{n\rightarrow \infty}\frac{z_{2n+1}}{w_{2n+1}}=\frac{aw_0+a\sum_{j=1}^\infty a^{-j}f(z_{2j-2})}{z_0+\sum_{j=1}^\infty a^{-j}f(z_{2j-1})}=\frac{a}{h_1}.
\end{align} 
Both the numerator and the denominator in $h_1,h_2$ are nonconstant holomorphic functions from $W$ to $\C$, indeed, by taking  two points $(z_0,w_0),(z_0', w_0')\in W$ with $|z_0-z_0'|, |w_0-w_0'|>2\Delta$ we have that $h_i(z,w)\neq h_i(z',w')$. So $h_1$ and $h_2$ are holomorphic functions from $W$ to $\hat{\C}$.
 
 We now show that $h_1,h_2$ are nonconstant; By Sard's Theorem and since $h_i(W)$ is contained in the line at infinity, $h_1,h_2$ have generic rank 1. Since $h_1=\frac{a}{h_2}$, this also implies that $h_1
 \neq h_2$.  Suppose for a contradiction that 
 $|h_1|=c$ is constant. 
  Then one has:
$$
|z_0|-\Delta\leq \left|z_0+\sum_{j=1}^\infty a^{-j}f(z_{2j-1})\right|=c\left|w_0+\sum_{j=1}^\infty a^{-j}f(z_{2j-2})\right|\leq c|w_0|+c\Delta,
$$
hence 
$$
|z_0|\leq c|w_0|+(c+1)\Delta,
$$

contradicting the fact that $(z_0,w_0)$ could be any point in   $W$, which is unbounded in the $z$ direction for any choice of $w$.
%
%
%
     
 \end{proof}

\subsection{Image of the limit functions $h_1, h_2$}
In this section we show that the image of  the limit functions $h_1, h_2$ contains the line at infinity minus $0,\infty$.
\begin{prop}\label{prop:image of h1} Let $F,U, h_1, h_2$ be as in Proposition~\ref{prop:distinct limit functions}.  Then
$$
h_1(U), h_2(U)\supset \ell^\infty\setminus\{ [0:1:0], [1:0:0]\}.
$$
\end{prop}

The idea of the proof is to show that $h_1$ is close enough to the model function $h_0(z,w):=\frac{z}{w}$ on suitable disks contained in $W$, and then to use the fact that $h_0(W)$ satisfies the claim together with Rouch\'e's Theorem to deduce the claim for $h_1$. The claim for $h_2$ follows because $h_2=\frac{a}{h_1}$.

\begin{thm}[Rouch\'e's Theorem] Let $D\subset\C$ be a Jordan  domain, $f,g$ be holomorphic in a neighborhood of $\ov{D}$.  Assume that $c\in g(D) $ and that 
$$
|f-g|<\dist\left(c,g(\partial D)\right) \qquad {\textrm on}\ \partial D.
$$
Then $c\in f(D)$.
\end{thm}
  
Observe that, for $c\in\C$, the complex line $L_c=\{(cw,w):\ w\in\C\}$  passing through the origin is mapped to the point $c$ under the map $h_0(z,w)=\frac{z}{w}$. Similarly,  the preimage of   $c=\infty$ under  $h_0$ is the line $L_\infty:=(\{(z,0):\ z\in\C\})$. We first need a lemma about the size of disks contained in $W$ whose center is a point $(c w_0,w_0)\in L_c$ and which is contained in a line 
 orthogonal to $L_c$, that is, a line of the form $\{(c w_0,w_0)+(-w,\overline c w)\}$ with $w\in\C$.  Let $\D_\delta$ denote  the Euclidean  disk of radius $\delta$ centered at the origin.
\begin{lem}\label{lem:maximal disk} For $c\in\C$,  $\delta>0$, and $(c w_0, w_0)\in W$ let $D_{c,\delta}(w_0)$ be the disk of radius $\sqrt{1+|c|^2} \delta$ defined as 
$$D_{c,\delta}(w_0)=\left\{(z,w)\in\C^2: (z,w)=(c w_0,w_0)+t(-1,\overline c), t\in\D_\delta\right\}.$$ 
Then  $D_{c,\delta}(w_0)\subset W$ for 
$$\delta= \min\left(\frac{|\Re w_0-R|}{|c|}, |\Re c \Re w_0-\Im c\Im w_0-R|\right)\,.$$
\end{lem}

\begin{proof}
A point $(z,w)\in\partial W$ satisfies either $\Re w= R$ or $\Re z=R$, so to find the maximal $\delta$ such that $D_{c,\delta}(w_0)\subset W$ we   impose the conditions
\begin{equation}
\begin{split}
  \Re \left(c w_0-t\right)&=R \\
  \Re\left(w_0+t \overline{c}\right)&=R
\end{split}\nonumber
\end{equation}
and find the minimal $\delta$ for which one is verified for some  $|t|=\delta$.  The above equations are equivalent to  
\begin{equation}
\begin{split}
    \Re c\Re w_0- \Im c \Im w_0 - \Re t   &=R \\
\Re w_0+ \Re c\Re t+\Im c\Im t &=R.
\end{split}\nonumber
\end{equation}
By setting $\Re t=x, \Im t=y$, finding the minimal $\delta$ is equivalent to finding the distance in $\R^2$ from the origin of two lines of the form
$$
Ax + By+C=0. 
$$
For the first equation $A=-1, B=0, C=\Re c\Re w_0-\Im c\Im w_0-R$. 
The distance of such a line from the origin is given by 

$$
\frac{|C|}{\sqrt{A^2+B^2}}
=  |\Re c\Re w_0-\Im c\Im w_0-R|. 
$$
For the second equation $A= \Re c$, $B=\Im c$, $C=\Re w_0-R$, and the distance of such a line from the origin is
$$\frac{|C|}{\sqrt{A^2+B^2}}
=\frac{|\Re w_0-R|}{|c|}$$
hence the theses.  
\end{proof} 
\begin{proof}[Proof of Proposition~\ref{prop:image of h1}] 
We will show that 
$h_1(W), h_2(W)\supset \ell^\infty\setminus\{0,\infty\}$, which implies the claim. 
Let 
$h_0(z,w):= \frac{z}{w}.$ It is easy to check that $h_0(W)=\ell^\infty$. We use this fact to show that for any  $c\in\hat{\C}\setminus\{ 0, \infty\}$, $c\in h_1(W)$.

In view of Rouch\'e's Theorem it is enough to find $r>0$ and a    one-dimensional disk $D\subset W$ such that 
\begin{itemize}
\item $h_0(D)$ contains a disk of radius $r$ centered at $c$ in $\ell^\infty\setminus \{0,\infty\}$;
 \item $|h_1- h_0|<r$ on $\partial D$. 
\end{itemize} 
Let $$w_0=M+R+i\frac{(M+R)\Re c-2M-R}{\Im c}\ \ \textrm{if}\ \ \Im c\neq 0$$ or $$w_0=2M+R\ \ \textrm{if}\ \ \Im c=0.$$
We claim that for  $M>0$  sufficiently large, the disk $D:= D_{c,\delta}(w_0)$  centered in $(cw_0,w_0)$  with   $\delta= (M-1)\min(|c|,\frac{1}{|c|})$   is contained in $W$ and satisfies the  requirements. 

We first check that  $D\subset W$. In view of Lemma~\ref{lem:maximal disk} we only need to check that
$$
\delta\leq \min\left(\frac{|\Re w_0-R|}{|c|}, \left|\Re c \Re w_0-\Im c\Im w_0-R\right|\right).
$$

If $\Im c\neq0$,  $\frac{|\Re w_0-R|}{|c|}=\frac{M}{|c|}>\delta$ and $|\Re c\Re w_0-\Im c\Im w_0-R|=2M>\delta$ for all choices of $M>0$. 

If $\Im c= 0$,    $\frac{|\Re w_0-R|}{|c|}=\frac{2M}{|c|}>\delta$,  and $|\Re c\Re w_0-\Im c\Im w_0-R|=|(2M+R)\Re c-R|\geq  (2M+R)|\Re c|-R\geq \delta$ for $M$ large enough.

From now on, it is no longer necessary to divide the two cases. We now compute the distance  $|h_0(\partial D)-c|.$ Let $t\in\C$, $|t|=\delta$ and $(z,w)=(c w_0, w_0)+t(-1, \ov{c})\in\partial D$. Then
$$
\left|h_0(z,w)-c\right|=\left| \frac{c w_0- t}{w_0+\ov{c}t} -c\frac{w_0+\ov{c}t}{w_0+\ov{c}t}\right|=\frac{|t|\left(1+|c|^2\right)}{\left|w_0+\ov{c}t\right|}=:A.
$$
We want to compare this with  $|h_0-h_1|$ on $\partial D$. Let us define
$$
k_1=k_1(z,w):= \sum_{j=1}^\infty a^{-j}f(z_{2j-1}),\qquad k_2=k_2(z,w):=\sum_{j=1}^\infty a^{-j}f(z_{2j-2}),
$$
 and note that $|k_1|,|k_2|$  are bounded uniformly in  $W$ (see (\ref{eqtn:DeltaBound})). Let $(z,w)\in \partial D$ as before. Then
$$
 \left|(h_1-h_0)(z,w)\right| = \left|\frac{z+k_1}{w+k_2}  -\frac{z}{w}  \right|               
= \left|\frac{k_1w-k_2 z}{w(w+k_2)}\right|=\frac{\left|(k_1-ck_2)w_0+(\ov{c}k_1-k_2)t\right|}{\left|w_0+\ov{c}t\right|\left|w_0+\ov{c}t+k_2\right|}=:B. 
$$
Calculating the ratio
$$
\frac{A}{B}=\frac{|t|\left(1+|c|^2\right)\left|w_0+\ov{c}t+k_2\right|}{\left|(k_1-ck_2)w_0+(\ov{c}k_1-k_2)t\right|}\ra\infty \quad\textrm{ as}\  M\ra\infty,
$$
since the numerator is a polynomial of degree $2$ in $M$ and the denominator is a polynomial of degree  $1$ in $M$ (indeed, both $|t|$ and $\Re w_0$ grow linearly in $M$). This implies that for $M$ large enough, $D$ satisfies the requirement for  Rouch\'e's Theorem, and hence $h_1(W)\supset \ell^\infty\setminus \{[1:0:0], [0:1:0]\}$. Since   $h_2=\frac{a}{h_1}$, the same holds for $h_2$.
\end{proof}

\begin{rem}
 If $R>\sup_{W_R}|\Delta|$, then we have precisely that $h_i(W_R)=\ell^\infty\setminus\{[0:1:0], [1:0:0]\}$, because the numerator and the denominator in (\ref{eqtn:definition of h1}) cannot attain the exact value $0$. If  $W_R$ is an absorbing domain for $U$ as in Section~\ref{sect:absorbing}, then  $h_i(U)=\ell^\infty\setminus\{[0:1:0], [1:0:0]\}$.
\end{rem}

\subsection{Conjugacy of $F$ to a linear map}\label{sect:conjugacy}
  
 \begin{prop}Let $F$ be as in Proposition~\ref{prop:distinct limit functions}. Then for $R$ sufficiently large $F$ is conjugate to the linear map $L(z,w)=(aw,z)$ on the set $\bigcup_{n\geq0} F^{-n}(W)$, where $W=W(R)=\{(z,w)\in\C^2: \Re z,\Re w>R\}$.
 \end{prop}
 \begin{proof}
 Recall that $|f(z)|$ is bounded by some constant, say $M$, for $R$ large enough and $\Re z> R$. Let $W:=W(R)$ for such $R$.  It is easy to check that $L^{-n}(z,w)=(\frac{z}{a^{n/2}},\frac{w}{a^{n/2}})$ if $n$ is even and $L^{-n}(z,w)=(\frac{w}{a^{(n-1)/2}},\frac{z}{a^{(n+1)/2}})$ if $n$ is odd; hence by  a direct computation,  using the fact that $a>1$,  for any $n\in\N$ and for any $P\in \C^2$ we have that  $\|L^{-n}(P)\|\leq a^{-\frac{n-1}{2}}\|P\|$.  
 
 Let $\phi_n:\C^2\ra\C^2$ be the automorphisms defined as
  $$
\phi_n:=L^{-n}\circ F^n. 
 $$
We will show that the  $\phi_n$ converge to a map $\phi:\C^2\ra\C^2$   uniformly on  $W$. 
 Since the $\phi_n$ satisfy the functional equation $\phi_{n+1}=L^{-1}\circ\phi_n\circ F$, %
  the map $\phi$ is a conjugacy between $F$ and $L$.

 Using the explicit expressions for the iterates of $F$ given by  (\ref{eqtn:recursive even}) and (\ref{eqtn:recursive odd})  we compute 
 \begin{align}
 \phi_{2k}(z,w)&=\left(z+\sum_{j=1}^ka^{-j}f(z_{2j-1}),w+\sum_{j=1}^ka^{-j}f(z_{2j-2})\right),\\
\phi_{2k+1}(z,w)&=\left(z+\sum_{j=1}^ka^{-j}f(z_{2j-1}),w+\sum_{j=1}^{k+1}a^{-j}f(z_{2j-2})\right),
\end{align}
 and taking the limit we obtain 
 $$
 \phi (z,w)=\left(z+\sum_{j=1}^\infty a^{-j}f(z_{2j-1}),w+\sum_{j=1}^\infty a^{-j}f(z_{2j-2})\right),
 $$
which is a biholomorphism between $W$ and $\phi(W)$ since both series converge because $f$ is bounded in a right half plane. It is injective by Hurwitz Theorem because the maps $\phi_n$ are injective  and their limit has rank 2 (see \cite{Krantz}, Exercise 3 on page 310). 
  
  For any $P\in F^{-k}(W)$ we extend $\phi$ as $\phi(P)= L^{-k}\circ \phi\circ F^k(P)
   $. Since $F$ is an automorphism and since $\phi\circ F=L\circ\phi$, the extension of  $\phi$ (which we still denote by $\phi$)  is well defined as a biholomorphism  from  $\bigcup_{n\geq0} F^{-n}(W)$ to  $\bigcup_{n\geq0} L^{-n}(\phi(W))$. 
  \end{proof}  
  
 \begin{rem}
 If $|f(z)|\ra0$ as $\Re z\ra\infty$, instead of just being bounded, and since the real parts of $z,w$ are increasing under iteration,  we have that $\phi(z,w)$ tends to the identity as $\Re z, \Re w \ra\infty$. However in general this may not be the case.
 \end{rem}
  
\subsection{Geometric structure of $U$ for $f(z)=e^{-z}$}\label{sect:absorbing}
In this section we prove that, in the special case that $f(z)=e^{-z}$, the Fatou component  $U$ is the union of the backwards images of $W$. As a corollary, using the linearization results from Section~\ref{sect:conjugacy} we obtain that $U$ is biholomorphic to $\H\times\H$.  In fact, the proof holds for any $f$ satisfying the hypothesis of Theorem~\ref{prop:absorbing domain} as long as $|f(z)|$ grows fast enough for $\Re z\ra-\infty$. It is based on a modification of the plurisubharmonic method used in \cite{henon1}, Section 5.
\begin{prop}\label{prop:absorbing domain}Let $F(z,w)=(aw+e^{-z},z)$, and $ U$ be as in Proposition~\ref{prop:distinct limit functions}. For $R$ sufficiently large and  $W=\{(z,w)\in\C:\Re z,\Re w>R\}$, the set $W$ is an absorbing domain for $U$, that is
$$U=A\ := \ \bigcup_{n\in\mathbb N}F^{-n}(W).$$
\end{prop}
The proof goes by contradiction, by assuming that there is  a point  $P\in U\setminus A$. 
We first show that we can assume that $P\in U\cap\partial A$ and that $h_1(P)\neq 0, \infty$.
 \begin{lem}\label{lem:Good P}If $U\neq A$, there exists $P\in \partial A\cap U$ such that $h_1(P)\neq 0, \infty$. 
 \end{lem}
 \begin{proof}
 If $U\neq A, $ since $U\supset A $ is connected and $A$ is open  we have that $\partial A\cap U\neq\emptyset$. 
 Since $U$ is a Fatou component,  the function $h_1$ is well defined on all of $U$ (though the numerator and the denumerator in expression (\ref{eqtn:definition of h1}) may not necessarily converge independently).  
 Let $\ZZ$ be the subset of $U$ such that $h_1$ takes the value $0$, and  $\PP$ the subset of $U$ such that $h_1$ takes the value $\infty$. 
 Suppose for the sake of contradiction that $\partial A \cap U$ is a subset of  $ \XX$. For  $P\in \partial A \cap U$  consider a neighborhood $V\subset U$ of $P$.  
Since $\XX$ is an analytic set we have that   $V\setminus\XX$ is connected (see for example Proposition 7.4 in \cite{KaupKaup}). Since  $P\in\partial A$ and $A$ is open, we have that $A\cap V\neq\emptyset$. Since $V\setminus\XX$ is connected and    $\partial A\cap U\subset \XX  $  by the contradiction assumption  we get that   $V\setminus \XX \subset A$.  Since $h_1$ is not constant, $\XX:=\ZZ\cup  \PP$ is  locally a finite union of complex curves and of finitely many points (see e.g. Section 5.1 and 5.2 in \cite{Chirka}). It  follows that there are infinitely many  directions such that a sufficiently small Euclidean disk $D$ tangent to that direction satisfies  $D\setminus\{P\}\subset A$. 

We now show that the existence of such  $D$ implies that $P\in A$. Indeed,  the sequence of harmonic functions $g_n:(z,w)\ra \Re z_n=\Re \pi_z(F^n(z,w))$  converges to infinity on compact subsets of $A$, hence, since  $D\setminus\{P\}\subset A$, it   converges  to infinity uniformly  on the boundary of a subdisk of $D$, hence converges to infinity on its center $P$ by Cauchy's formula. Hence the real parts of the first coordinate of iterates of $P$ converge to infinity, as well as the real parts of the second coordinate (since $w_n=z_{n-1}$), which implies $P\in A$. This contradicts the fact  that $P\in U\setminus A$, and hence the assumption that $(\partial A\cap U)\subset \XX$ is false and there is a point $P$ as in the claim.
 \end{proof}
 From now on we consider $P$ with the properties of Lemma~\ref{lem:Good P}.  
Since $h_1\neq 0, \infty$ in $P$, the same is true for $h_2=\frac{a}{h_1}$, and hence by continuity $h_1, h_2$ do not take the values $0, \infty$ in some small closed  ball $B$ centered at $P$.

Hence  we can define the quantities 
\begin{equation}
\begin{split}
M&:=\max_{B}(\max(|h_1|,|h_2|))<\infty\\
m&:=\min_{B}(\min(|h_1|,|h_2|))>0.\\
\end{split}\nonumber
\end{equation}
Note that  $M>1$ because $h_2=\frac{a}{h_1}$ and $a>1$. 
By Corollary~\ref{cor:slow growth 2} if $0<\epsilon<m$ there exists a constant $C$ such that for every $P=(z_0,w_0)\in B$, 
 \begin{equation}\label{eqtn:sandwich growth}
 |z_n|\leq C(M+\epsilon)^n.
 \end{equation}
 Recall that $w_n=z_{n-1}$ to get 
 \begin{equation}\label{eqtn:sandwich growth2}
  |w_n|\leq C(M+\epsilon)^{n-1}.
 \end{equation}

 The proof of Proposition~\ref{prop:absorbing domain} relies on the following technical lemma. 
\begin{lem}\label{lem:harmonic functions} 
Define the sequence of harmonic functions $u_n$ from a neighborhood of $B$ to $\R$ as $u_n(z):=\frac{-\Re z_n}{n}$. Then 
\begin{enumerate}
\item   $u_n\leq \log M$ in $U$ for $n$ large enough;
\item $u_n\ra -\infty$  uniformly on compact subsets of  $A$; 
\item Let $P\in U\setminus A$. Then for every $\epsilon>0$ there exists a subsequence $n_k$ such that $u_{n_k}(P)>-\epsilon$.
 \end{enumerate}
\end{lem}
\begin{proof}
\begin{enumerate}
\item Suppose that  there is a subsequence $(n_k)$ and points $(z^k,w^k)\in B$ such that 
$$\frac{-\Re z^k_{n_k}}{n_k}>\beta$$
for some $\beta$ and  let us show that $\log M$ is an upper bound for $\beta$.  
It follows  that $\Re z^k_{n_k}<-n_k\beta$. Since $f(z)=e^{-z}$ we have (using the triangular inequality in the second step and (\ref{eqtn:sandwich growth}) for  $w^k_{n_k}$ in the third step) that 
$$
|z^k_{n_k+1}|=|e^{-z^k_{n_k}}+a w^k_{n_k}|\geq  |e^{-z^k_{n_k}}|-|a w^k_{n_k}|> e^{+n_k\beta}-a C(M+\epsilon)^{n_k-1}.
$$
On the other hand, again using (\ref{eqtn:sandwich growth2}), we have that 
$$ |z^k_{n_k+1}|\leq C(M+\epsilon)^{n_k}
$$ 
hence 
$$
 e^{n_k\beta}-a C(M+\epsilon)^{n_k-1}<C(M+\epsilon)^{n_k}
$$
which gives (using $M>1$) 
$$
e^{n_k\beta}<(a+1) C (M+\epsilon)^{n_k}
$$
from which (using $n_k\ra\infty$ and $\epsilon\ra0$) we obtain $\beta\leq\log M$. 
\item  Let $K$ be a compact subset of $A$. Then there exists $n$ such that $F^n(K)\subset W$, so it is enough to show the claim for a compact subset of $W$.  By the explicit expression for $z_{2n}, z_{2n+1}$ given by (\ref{eqtn:recursive even}), (\ref{eqtn:recursive odd}) we get that $\Re z_n \geq a^n (R- \Delta)$ as defined in (\ref{eqtn:Delta}). Hence by assuming that $R$ is chosen large enough so that $R- \Delta>0$, we get that 
$$u_n\leq\frac{-a^n (R-\Delta)}{n}\ra-\infty.$$
\item 
If not, there exists $\epsilon>0$, $N\in\N$ such that 
$$
u_n(P)\leq -\epsilon \qquad\text{for all}\ n\geq N.
$$ 
Hence if $F^{n}(P)=(z_n,w_n)$ we have that  $\frac{-\Re z_n}{n}\leq-\epsilon$ for all $n\geq N$, so  $\Re z_n\geq n\epsilon>R$ for $n$ large since $\epsilon>0$. Since $w_n=z_{n-1}$, $(z_n,w_n)\in W$ and $(z_0, w_0)\in F^{-n}(W)\subset A$. 
\end{enumerate}
\end{proof}
\begin{rem} We only use the assumption  $f(z)=e^{-z}$ to prove 1., that is, that  the $u_n$ are bounded from above. In fact, it is enough to assume that   $|f(z)|$  grows sufficiently fast as $\Re z\ra -\infty$.
\end{rem}
\begin{proof}[Proof of Proposition~\ref{prop:absorbing domain}] Let $P$ as in Lemma~\ref{lem:Good P}, $B$ be a ball centered in $P$ as described above, and let $D$ be a one-dimensional Euclidean disk  compactly contained in $B$, intersecting $A$,  and passing through $P$.  
 Consider the real one-dimensional  Lebesgue measure on $\partial D$.  
Let $K $ be a compact subset of $A$ such that the  measure in $\partial D$ of $K\cap \partial D$ is strictly positive.  This can be done because $A$ is open, hence $A\cap\partial D$ is open in the topology of $\partial D$. Let $\mu_{\text{good}}>0$ be the measure of the set $\partial D\cap K$ and $\mu_{\text{bad}}$  be the measure of the set $\partial D\cap (U\setminus K)$. Since $U$ contains $B$,  $\partial D=(\partial D\cap K )\cup (\partial D\cap (U\setminus K))$, and since $K$ is compact and $U$ is  open, the sets in question are measurable.  

By Lemma~\ref{lem:harmonic functions} for any given $\MM>0$ there exists $N$ such that $u_N\leq -\MM$ on $K$, $u_N(P)\geq -\epsilon$ for $\epsilon$ arbitrarily small  since $P\in U\setminus A$, and $u_N(P)\leq \log M$ on $U$. By the mean value property for $u_N$ we have 
$$
-\epsilon\leq u_n(P)=\!\!\int_{\partial D} u_N(\zeta)d\zeta=\!\!\int_{\partial D\cap K} u_N(\zeta)d\zeta+\!\int_{\partial D\cap (U\setminus K) } u_N(\zeta)d\zeta\leq -\MM \mu_{\text{good}}+ \log M \mu_{\text{bad}}.
$$
Since $\MM$ is arbitrarily large, this gives  a contradiction.
\end{proof}

 \begin{cor}Let $F(z,w)=(aw+e^{-z},z)$, and $ U$ be as in Proposition~\ref{prop:distinct limit functions}.  Then $U$ is biholomorphic to  $\{(z,w)\in\C: \Re z,\Re w> 0\}.$
 \end{cor}
\begin{proof} Let $W_R=\{(z,w)\in\C: \Re z,\Re w>R \}$ with $R$ large enough so that Proposition~\ref{prop:absorbing domain} holds and so that  $R-\Delta>0$ with $\Delta $ defined as in (\ref{eqtn:Delta}), with $f(z)=e^{-z}$. Then by  Proposition~\ref{prop:absorbing domain} $W_R$ is an absorbing domain for $U$, and by the explicit form of $\phi$, 
\begin{equation}
\begin{split}
W_{R+\Delta}=\{(z,w)\in\C: \Re z,\Re w> R&+\Delta\}\subset\phi(W)\\
&\subset \{(z,w)\in\C: \Re z,\Re w> R-\Delta\}=W_{R-\Delta}.
\end{split}\nonumber
\end{equation}
Since $R\pm \Delta>0$ we have that $$\bigcup_n L^{-n}(W_{R+\Delta})=\bigcup_n L^{-n}(W_{R-\Delta})=\{(z,w)\in\C: \Re z,\Re w> 0\}.$$
It follows that
$$
\bigcup_n L^{-n}(\phi(W_R))=\left\{(z,w)\in\C: \Re z,\Re w> 0\right\}.
$$
Since $\phi$ is a biholomorphism between $\bigcup_n L^{-n}(\phi(W_R))$  and $\bigcup_n F^{-n}((W_R))=U$ the claim follows. 
\end{proof} 

 \bibliographystyle{amsalpha}
\bibliography{Baker}
\end{document}